%% file: logadd.tex
\documentclass[12pt]{amsart}
\usepackage{amsmath}
\usepackage{amssymb}
\usepackage{amsfonts}
\usepackage{amsthm}
\usepackage{verbatim}
\usepackage{amscd}
\usepackage{cite}
\usepackage{leftidx}
\usepackage{enumerate}
\usepackage{txfonts}
\usepackage{manfnt}
\usepackage{amscd}
\usepackage[mathscr]{eucal}
\usepackage{hyperref}
\usepackage{datetime2}
\usepackage{mathpazo}

\input{myheader3.tex}

\def\rmRes{\mathrm{Res}}

\newtheorem{thm}{Theorem} 

\newtheorem*{thm*}{Theorem}
\newtheorem*{prop*}{Proposition}
\newtheorem{cor}[thm]{Corollary}
\newtheorem*{cor*}{Corollary}

\newtheorem{lem}[thm]{Lemma}

\newtheorem*{claim*}{Claim}
\newtheorem{prop}[thm]{Proposition}

\theoremstyle{remark}

\newtheorem{rem}[thm]{Remark}
\newtheorem*{rem*}{Remark}
\newtheorem{crit-rem}[thm]{Critical remark}

\newtheorem{example}[thm]{Example}
\newtheorem*{example*}{Example}
\newtheorem*{lem*}{Lemma}

\newtheorem*{defn*}{Definition}

\begin{document} 
	\title{Complexes, Residues and obstructions\\  for
		log-symplectic manifolds }
	\author 
	{Ziv Ran}
	
	
	\date {\DTMnow}
	
	
	\address {\nl UC Math Dept. \nl
		Big Springs Road Surge Facility
		\nl
		Riverside CA 92521 US\nl 
		ziv.ran @  ucr.edu\nl
		\url{http://math.ucr.edu/~ziv/}
	}
	
	
	\begin{abstract}
		We consider compact K\"ahlerian manifolds $X$ 
		of even dimension 4 or more, endowed with
		a log-symplectic   structure 
		$\Phi$, a generically nondegenerate closed 2-form
		 with simple poles on a divisor $D$ with local normal
		crossings. A simple linear inequality involving the iterated 
		Poincar\'e  residues of $\Phi$
		at components of the double locus of $D$ 
		ensures that
		 the pair $(X, \Phi)$
		has unobstructed deformations and that $D$ deforms locally trivially. 
	
	\end{abstract}
	\maketitle
	\section*{Data availaibility statement}
	There is no data set associated with this paper.
	\section*{Introduction}
	A  log-symplectic manifold is a  pair consisting of 
	a complex manifold $X$, usually compact and K\"ahlerian,
	 together with a log-symplectic structure.
	A log-symplectic structure can be defined
	either as a generically nongegenerate meromorphic closed  2-form $\Phi$ with normal-crossing 
	(anticanonical) polar
	divisor $D$, or equivalently as a generically nondegenerate holomorphic tangential 2-vector $\Pi$
	such that $[\Pi,\Pi]=0$ with
	normal-crossing degeneracy divisor $D$. The two structures are related via $\Pi=\Phi\inv$.  
	See \cite{dufour-zung} or \cite{polishchuk-ag-poisson}
	or \cite{ciccoli}  or \cite{pym-constructions} for basic facts on 
	Poisson and log-symplectic manifolds 
	and \cite{ginzburg-kaledin} (especially the appendix),
	\cite{goto-rozanski-witten}, \cite{marcut-ot}, \cite{lima-pereira}
	or \cite{namikawa}, and references therein, for deformations.\par
	Understanding log-symplectic manifolds unavoidably involves understanding their deformations.
	In the very special case of \emph{symplectic} manifolds, where $D=0$, the classical theorem
	of Bogomolov \cite{bogomolov} shows that the pair $(X, \Phi)$ has unobstructed deformations.
	In \cite{logplus} we obtained a generalization of this result
	which holds when $\Phi$ satisfied a certain 'very general position' condition
	with respect to $D$ (the original statement is corrected in the subsequent
	erratum/corrigendum).
	Namely, we showed in this case that $(X, \Phi)$ has 'strongly unobstructed' deformations,
	in the sense that it has unobstructed deformations and $D$ deforms locally trivially.
	
	 Further results on unobstructed deformations
	(in the sense of Hitchin's generalized geometry \cite{hitchin-generalizedcy})
	and Torelli theorems in the case where $D$ has global normal crossings
	 were obtained by Matviichuk, Pym and Schedler \cite{pym-loctorelli},
	based on their notion of
	holonomicity.\par
	Our purpose here is to prove a more precise strong unobstructedness result compared to
	\cite{logplus}, nailing down the generality required: we will show in Theorem \ref{iso} that 
	strong unobstructedness can fail only when the log-symplectic
	structure $\Phi$, more precisely its (iterated Poincar\'e) residues at codimension-2 strata of the polar
	divisor $D$ (which are essentially the (locally constant) coefficients of $\Phi$ with respect to
	a suitable basis of the log forms adapted to $D$)
	 satisfy certain special linear relations with integer coefficients.
	 Explicitly, at a triple point of $D$ with branches labelled 1,2,3 and associated
	 residues $c_{12}, c_{23}, c_{31}$, the condition is
	 \[c_{23}+c_{31}\in\N c_{12}.\]
	 Essentially, if this never happens over the entire triple locus
	  then $(X, \Phi)$ has strongly unobstructed deformations.

\par
	The strategy of the proof as in \cite{logplus} is to study the inclusion of complexes
	\[(T^\bullet_X\mlog{D}, [\ .\ ,\Pi])\to (T^\bullet_X, [\ .\ ,\Pi]) \ ,\]
	albeit from a more global viewpoint.
	In fact as in \cite{logplus} 
	it turns out to be more convenient to transport the situation over to the De Rham
	side where it becomes an inclusion
	\[(\Omega^\bullet_X\llog{D}, d)\to (\Omega^\bullet_X\llogp{D}, d)\]
	where the latter 'log-plus' complex is a certain complex of meromorphic forms with
	poles on $D$. We study a filtration, introduced in \cite{logplus}, 
	interpolating between the two complexes,
	especially its first two graded pieces. As we show, the first piece is automatically
	exact, while 0-acyclicity for the second piece leads to the above cocycle condition.
	 See  \S \ref{comparing} for details.\par
	We begin the paper with a couple of auxiliary, independent sections. In \S \ref{principal} 
	we construct a 'principal
	parts complex' associated to an invertible sheaf $L$ on a smooth variety, 
	extending the principal parts
	sheaf $P(L)$ together with the universal derivation $L\to P(L)$. 
	We show this complex is always exact. In
	\S\ref{calculus} we show that, for any normal-crossing divisor $D\subset X$
	on any smooth variety, the log complex
	$\Omega^\bullet_X\llog{D}$- unlike $\Omega_X^\bullet$ itself-
	 can be pulled back to a complex of vector bundles on the
	normalization of $D$. These complexes play a role in our analysis of the 
	aforementioned inclusion map.\par
	I am grateful to Brent Pym for helpful communications, in particular for communicating Example
	\ref{pym-example}.
	
\section{Principal parts complex}\label{principal}
In this section $X$ denotes an arbitrary $n$-dimensional
smooth complex variety and $L$ denotes an invertible sheaf on $X$.
\subsection{Principal parts}
 The Grothendieck principal
parts sheaf $P(L)$ (see EGA) is a rank-$(n+1)$ bundle on $X$ defined as
\[P(L)=p_{1*}(p_2^*L\otimes(\O_{X\times X}/\I_\Delta^2))\]
where $\Delta\subset X\times X$ is the diagonal and $p_1, p_2:X\times X\to X$
are the projections.
We have a short exact sequence
\[\exseq{\Omega^1_X\otimes L}{ P(L)}{ L}\]
whose corresponding extension class in $\Ext^1(L, \Omega^1_X\otimes L)=H^1(X, \Omega^1_X)$
coincides with $c_1(L)$. The sheaf 
\[P_0(L)=P(L)\otimes L\inv,\]
which likewise has extension class $c_1(L)$, is called
the \emph{normalized} principal parts sheaf.
The map $P(L)\to L$ admits a splitting $d_L:L\to P(L)$ that
is a derivation, i.e.
\[d_L(fu)=fd_Lu+df\otimes u.\]
In fact, $d_L$ the universal derivation on $L$. Moreover $P(L)$ is generated
over $\O_X$ by the image of $d_L$. Likewise, $P_0(L)$ is generated
by elements of the form $\dlog(u):=d_Lu\otimes u\inv$ where $u$ is a local generator of $L$.\par
\subsection{Complex}
It is well known that $P(L^{m+1})\simeq P(L)\otimes L^m, m\geq 0$ which in 
particular yields a derivation $L^{n+1}\to P(L)\otimes L^n, n\geq 0$.
In fact, This map extends to a complex
that we denote by $P_{n+1}^\bullet(L)$  or just
$P^\bullet(L)$ and call the ( $(n+1)$st) \emph{principal parts complex}
of $L$:
\eqspl{}{
P^\bullet(L): L^{n+1}\to P(L)L^{n}\to\wedge^2P(L)L^{n-1}\to...\wedge^{n+1}P(L)=	
\Omega^n_X\otimes L^{n+1}.
}
The differential is given, in terms of local $\O_X$-generators $u_1,...,u_k, v_1,...,v_\l$
of L, by
\[d(u_1...u_kd_L(v_1)\wedge...d_L(v_\l)
=\sum u_1...\hat{u_i}...u_kd_L(u_i)\wedge d_L(v_1)\wedge...\wedge d_L(v_\l)\]
and extending using additivity and the derivation property.
There are also similar shorter complexes
\[ L^m\to P(L)L^{m-1}\to...\to\wedge^mP(L).\]
Note the exact sequences
\[\exseq{\Omega^m_XL^m}{\wedge^mP(L)}{\Omega^{m-1}_XL^m}
.\]
	These sequences splits locally and also split globally
	 whenever $L$ is a flat line bundle. In such cases, we get a short
	 exact sequence
	\eqspl{}{
		\exseq{\Omega^\bullet_XL^{n+1}[-1]}{P^\bullet(L)}{\Omega^\bullet_XL^{n+1}}
	}
	The principal parts complex $P^\bullet(L)$ may be tensored with $L^{j-n-1}$, 
	for any $j> 0$, yielding
	the $j$-th principal parts complex:
\eqspl{}{
P^\bullet_j(L): L^j\to P_0(L)L^j\to \wedge^2 P_0(L)L^j\to...\to\wedge^{n+1}P_0(L)L^j	
} The differential is defined by setting
\[d(\dlog(u_1)\wedge...\wedge\dlog(u_i)v^j)=j\dlog(u_1)\wedge...\wedge\dlog(u_i)\dlog(v)v^j\]
where $u_1,...,u_i, v$ are local generators for $L$,
and extending by additivity and the derivation property. Thus, 
$P^\bullet(L)=P^\bullet_{n+1}(L)$.\par
An important property of principal parts complexes is the following:
\begin{prop}\label{exact}
	For any local system $S$, 
	the complexes $P^\bullet_j(L)\otimes S$ are  null-homotopic and exact for all $j>0$.
	\end{prop}
\begin{proof}
	The assertion being local, we may assume $L$ is trivial and $S=\C$ so the
	$i$-th term of $P^\bullet_j(L)\otimes S$ is just $\Omega^{i-1}_X\oplus\Omega^i_X$
	and the differential is $\left(\begin{matrix}d&\id\\0&d
	\end{matrix}\right).$ Then a homotopy is given by
	$\left(\begin{matrix}0&\id\\0&0
	\end{matrix}\right).$ Thus, $P^\bullet_j(L)$ is null-homotopic, hence exact.
	\end{proof}
\subsection{Log version}
The above constructions have an obvious extension to the log situation. Thus let $D$
be a divisor with normal crossings on $X$. We define $P(L)\llog{D}$ as the image of $P(L)$
under the inclusion $\Omega_X\to\Omega_X\llog{D}$, and likewise for $P_0(L)\llog{D}$. 
Then as above we get complexes 
\eqspl{}{P^\bullet_j(L)\llog{D}: L^j\to P_0(L)\llog{D}L^j\to...\to\wedge^{n+1}P_0(L)\llog{D}L^j.
}

\subsection{Foliated version}
Let $F\subset\Omega_X\llog{D}$ be an integrable subbundle
of rank $m$. Then $F$ gives rise to a foliated
De Rham complex $\wedge^\bullet(\Omega_X\llog{D}/F)$, we well as a foliated principal
parts sheaf $P^1_F(L)\llog{D}=P^1(L)\llog{D}/F\otimes L$. 
Putting these together, we obtain the foliated principal parts complexes
(where $P_{0,F}(L)\llog{D}:=P_0(L)\llog{D}/F$):
\eqspl{}{
P^\bullet_{j, F}(L)\llog{D}: L^j\to P_{0, F}(L)\llog{D}L^j\to...\to \wedge^{n-m+1}P_{0, F}(L)\llog{D}	
}
Note that the proof of Proposition \ref{exact} made no use of
of the acyclicity of the De Rham complex. Hence the same proof applies verbatim to yield
\begin{prop}\label{exact-log}
For any local system $S$,
	the complexes $P^\bullet_{j, F}(L)\llog{D}\otimes S$ are  
	null-homotopic and exact for all $j>0$.
	
\end{prop}

\section{Calculus on normal crossing divisors}\label{calculus}
In this section $X$ denotes a smooth variety or complex manifold and  $D$ denotes
 a locally normal-crossing divisor on  $X$.
Our aim is to show that the log complex on $X$, unlike its De Rham analogue,
can be pulled back to the normalization of $D$.
 \par
\subsection{Branch normal} 
Let $f_i:X_i\to X$ be the normalization of the $i$-fold locus of $D$.
A point on $X_i$ consists of a point on $D$ together with a choice 
of $i$ distinct local branches of $D$ at it.
There is a canonical induced normal-crossing divisor $D_i$ on $X_i$:
at a point where $x_1...x_m$ is an equation for $D$ and $x_1,...,x_i$
are the chosen branches, the equation of $D_i$ is $x_{i+1}...x_m$. 
Note the exact sequence
\eqspl{}{
\exseq{T_X\mlog{D}}{T_X}{f_{1*}N_{f_1}}	
}
where $N_{f_1}$ is the normal bundle to $f_1$ which fits in an exact sequence
\[\exseq{T_{X_1}}{f^*_1T_X}{N_{f_1}}.\]
Locally, $N_{f_1}$ coincides with $x_1\inv\O_X/\O_X$ where $x_1$ is a 'branch equation':
to be precise, if $K$ denotes the kernel of the natural surjection $f_1\inv\O_X\to\O_{X_1}$,
	then $J=K/K^2=K\otimes_{f\inv\O_X}\O_{X_1}$ is an invertible $\O_{X_1}$-module 
	locally generated by $x_1$ and
	$N_{f_1}=J\inv$. Note that
	\[N_{f_1}\otimes\O_{X_1}(D_1)=f_1^*(\O_X(D)).\]
	\par
\subsection{Pulling back log complexes}
Interestingly, even though
the differential on the pullback De Rham complex
$f_1\inv\Omega_X^\bullet$ does not extend to $f\inv\Omega_X^\bullet\otimes\O_{X_1}$,
the analogous assertion for the log complex does hold: the differential on
 $f_1\inv\Omega_X^\bullet\llog{D}$
extends to what might be called the restricted log complex:
\[f_1^*\Omega_X^\bullet\llog{D}=f_1\inv\Omega_X^\bullet\llog{D}\otimes\O_{X_1}.\]
This is due to the identity (where $x_1$ denotes a branch equation)
\[dx_1=x_1\dlog(x_1).\] 

Note that the residue map yields an exact sequence
\eqspl{log-res0}{
	0\to\Omega^1_{X_1}\llog{D_1}\stackrel{j}{\to} f_1^*\Omega^1_X\llog{D}\stackrel{\rmRes}{\to} \O_{X_1}\to 0.	
}
Note that the rsidue map commutes with exterior derivative. Therefore this sequence
induces a short exact sequence of complexes
\eqspl{}{
\exseq{\Omega^\bullet_{X_1}\llog{D_1}}{f_1^*\Omega^\bullet _X\llog{D}}{\Omega_{X_1}^\bullet{\llog{D_1}}[-1]}.	
}
Furthermore, a twisted form of the restricted log complex, called the normal log complex,
also exists:
\eqspl{Nf1}{
N_{f_1}\otimes f_1^*\Omega^\bullet_X\llog{D}	:N_{f_1}\to N_{f_1}\otimes f_1^*\Omega^1_X\llog{D}
\to...
}
this is thanks to the identity, where $\omega$ is any log form,
\[d(\omega/x_1)=(d\omega)/x_1-\dlog(x_1)\wedge\omega/x_1.\]
Now recall the exact sequence coming from the residue map
\[\exseq{\Omega_{X_1}\llog{D_1}}{f_1^*\Omega_X\llog{D}}{\O_{X_1}}\]
In fact, it is easy to check that this exact sequence has extension class
$c_1(N_{f_1})$ hence identifies $f_1^*\Omega_X\llog{D}$ with $P_0(N_{f_1})$ so that
	the normal log complex \eqref{Nf1} may be identified with the principal parts
	complex $P^\bullet(N_{f_1})$:
	\begin{lem}\label{normal-log-complex-lem}
	The normal log complex $N_{f_1}\otimes f_1^*\Omega_X\llog{D}$ is 
	isomorphic to $P^\bullet(N_{f_1})$, hence is exact.
	\end{lem}
Similarly, a pull back log complex $f_k^*\Omega^\bullet_X\llog{D}=f_k\inv\Omega_X^\bullet\llog{D}\otimes\O_{X_k}$
exists for all $k\geq 1$.
A similar
twisted log complex also exists
the determinant of the normal bundle $N_{f_k}$:
\eqspl{}{
	\det N_{f_k}\otimes f_k^*\Omega_X^\bullet\llog{D}:
	\det N_{f_k}\to \det N_{f_k}\otimes\Omega^1_X\llog{D}\to...
}
This comes from (where $x_1,...,x_k$ are the branch equations at a given point of $X_k$):
\[d(\omega/x_1...x_k)=d\omega/x_1...x_k-\dlog(x_1...x_k)\omega/x_1...x_k).\]
\subsection{Iterated residue}
We have a short exact sequence of vector bundles on $X_k$:
\eqspl{}{
\exseq{\Omega_{X_k}\llog{D_k}}{f_k^*\Omega_X\llog{D}}{\nu_k\otimes\O_{X_k}}	
}
where $\nu_k$ is the local system of branches of $D$ along $X_k$ and the right map
is multiple residue. Taking exterior powers, we get various exact Eagon-Northcott complexes.
In particular, we get surjections, called iterated Poinca\'e residue:
\eqspl{}{
f_k^*\Omega_X^i\llog{D}\to \Omega_{X_k}^{i-k}\llog{D_k}\otimes {\det}_{\C}(\nu_k), i\geq k,	
}
\eqspl{}{
f_k^*\Omega_X^i\llog{D}\to\wedge^i_\C\nu_k\otimes \O_{X_k}, i\leq k.	
}
${\det}_\C(\nu_k)$ is a rank-1 local system on $X_k$ which may be called the 'normal
orientation sheaf'. The maps for $i\geq k$ together yield a surjection
\eqspl{}{
f_k^*\Omega_X^\bullet\llog{D}\to \Omega_{X_k}^\bullet\llog{D_k}[-k]\otimes\det(\nu_k).	
}
\section{Comparing log and log plus complexes}\label{comparing}
In this section $X$ denotes a log-symplectic smooth variety with log-symplectic form 
$\Phi$ and corresponding Poisson vector $\Pi=\Phi\inv$, and $D$ denotes the degeneracy divisor
of $\Pi$ or polar divisor of $\Phi$. Our aim is to prove Theorem \ref{iso} which
shows that deformations of $(X, \Phi)$ coincide with locally trivial deformations
of $(X, \Phi, D)$ and are unobstructed.\par 
\subsection{Setting up}
We will use $\Omega_X^{+\bullet}$ to denote $\bigoplus\limits_{i>0}\Omega^i_X$
and similarly for the log versions. This to match with the Lichnerowicz-Poisson complex
$T^\bullet_X$ and $T^\bullet_X\mlog{D}$. Thus, interior multiplication by $\Phi$ induces
and isomorphism $T^\bullet_X\mlog{D}\to\Omega^\bullet_X\llog{D}$. Equivalently, $\Phi$ itself
is a form im $\Omega^2_X\llog{D}$ inducing a nondegenerate pairing on $T_X\mlog{D}$.
In terms of local coordinates, at a point of multiplicity $m$ on $D$,
we have a basis for $\Omega_X\llog{D}$ of the form
\[\eta_1=\dlog(x_1),...,\eta_m=\dlog(x_m), \eta_{m+1}=\dlog(x_{m+1}),...\]
and then
\[\Phi=\sum b_{ij}\eta_i\wedge\eta_j.\]
\par
We have an inclusion of complexes
\[T_X^\bullet\mlog{D}\to T^\bullet_X\]
where, for $X$ compact K\"ahler, the first complex controls 'locally trivial' deformations of
$(X, \Pi)$, i.e. deformations of $(X, \Pi)$ inducing a locally trivial deformation
of $D=[\Pi^n]$, and the second complex controls all deformations of $(X, \Pi)$.
It is known (see e.g. \cite{logplus})
 that locally trivial deformations of $(X, \Pi)$ are always unobstructed
and have an essentially Hodge-theoretic (hence topological) character, so one is
interested in conditions to ensure that the above inclusion induces an isomorphism 
on deformation spaces; as is well known, the latter would follow if one can show
that the cokernel of this inclusion has vanishing $\hh^1$. \par
Our approach to this question starts with the above 'multiplication by $\Phi$' isomorphism
\[(T_X^\bullet\mlog{D}, [\ .\ ,\Pi])\to (\Omega^{+\bullet}_X\llog{D}, d).\]
This isomorphism extends to an isomorphism to $T^\bullet_X$ with a certain subcomplex
of $\Omega^{+\bullet}_X(*D)$, the meromorphic forms regular off $D$, that we 
call the log plus complex and denote by
$\Omega_X^{+\bullet}\llogp{D}$.

Our goal then becomes that of comparing the log and log-plus complexes.
To this end we introduce a filtration on $\Omega^{+\bullet}_X\llogp{D}$, essentially
the filtration induced by the exact sequence
\[\exseq{T_X\mlog{D}}{T_X}{f_{1*}N_{f_1}}\]
and its isomorphic copy
\[\exseq{\Omega_X\llog{D}}{\Omega_X\llogp{D}}{f_*N_{f_1}}\]
where $f_1:X_1\to D\subset X$ is the normalization of $D$ and $N_{f_1}$ is the associated
normal bundle ('branch normal bundle'). We will show that the first graded piece is always
an exact complex. The second graded piece is much more subtle. We will show that it is locally
exact in degree 0 unless the log-symplectic form $\Phi$, i.e. the matrix $(b_{ij})$
above satisfies some special relations with
integer coefficients.

The computations of this section are all local in character, though the applications are global.

\subsection{Residues and duality}
Let $f_i:X_i\to X$ be the normalization of the $i$-fold locus of $D$, $D_i$ the induced
normal-crossing divisor on $X_i$. Thus a point of $X_i$ consists of a point $p$
of $D$ together with a choice of an unordered set $S$ of $i$ branches of $D$ through $p$
and $D_i$ is the union of the branches of $D$ not in $S$.
We consider first the codimension-1 situation.
As above, we have a residue exact sequence
\eqspl{log-res}{
0\to\Omega^1_{X_1}\llog{D_1}\stackrel{j}{\to} f_1^*\Omega^1_X\llog{D}\stackrel{\rmRes}{\to} \O_{X_1}\to 0	
} (the right-hand map given by residue is locally  evaluation on $x_1\del_{x_1}$
where $x_1$ is a local equation for the branch of $D$ through the given point of $X_1$
). Note that if $\eta$ comes from a closed form on $X$ near $D$ then $\rmRes(\eta)$ is a constant.\par
Dualizing \eqref{log-res}, we get
\eqspl{co-res}{
0\to\O_{X_1}\stackrel{\check{R_1}}{\to}f_1^*T_X\mlog{D}\stackrel{\check{j}}{\to}T_{X_1}\mlog{D_1}\to 0,
} where the left-hand map, the 'co-residue', is locally multiplication by $x_1\del_{x_1}$
where $x_1$ is a branch equation).  Set
\[v_1=x_1\del_{x_1}.\]
Then $v_1$ is canonical
as section of $f_1^*T_X\mlog{D}$ , independent
of the choice of local equation $x_1$. By contrast, $\del_{x_1}$ as section of $f_1^*T_X$ is canonical
only up to a tangential field to $X_1$, and generates $f_1^*T_{X}$ modulo $T_{X_1}\mlog{D}$.\par
Now $f_1^*\Omega^1_X\llog{D}$ and $f^*_1T_X\mlog{D}$ admit mutually inverse isomorphisms 
\[i_{X_1}\Pi:=\carets{\Pi, .}_{X_1}=f_1^*\carets{\Pi,.},
i_{X_1}\Phi:= \carets{ \Phi, .}_{X_1}=f_1^*\carets{\Phi, .}.\] The composite 
\[\check{j}\circ i_{X_1}\Pi\circ j:\Omega^1_{X_1}\llog{D_1}\to T_{X_1}\mlog{D_1}\]
  has a rank-1 kernel that is
 	the kernel of the Poisson vector on $X_1$ induced by $\Pi$, aka the conormal
 	to the symplectic foliation on $X_1$. 
 	Now set
 	\[\psi_1=i_{X_1}(\Phi)(v_1)=\carets{\Phi, v_1}_{X_1}.\]
 	Then $\psi_1$ is locally the form in $\Omega_{X_1}\llog{D_1}$
 	 denoted by $x_1\phi_1$ in \cite{logplus}. 
 	Again $\psi_1$ is canonically defined, independent of choices
 	and corresponds to the first column of the $B=(b_{ij})$ matrix for a local coordinate system
 	$x_1, x_2,...$ compatible with the normal-crossing divisor $D$.
By contrast, $\phi_1$,
 	which depends on the choice of local equation $x_1$, is canonical up to
 	a log form in $\Omega_{X_1}\llog{D_1}$ and generates $\Omega_{X_1}\llogp{D_1}$
 	modulo the latter.\par
 	In $X_1\setminus
 	D_1$, $\Phi$ is locally of the form $\dlog(x_1)\wedge dx_2+$(symplectic), 
 	so there $\psi_1=dx_2$. Note that by skew-symmetry we have
 	\[\rmRes\circ i_{X_1}(\Phi)\circ\check{R_1}=0.\]
 	Thus, locally $\psi_1\in\Omega_{X_1}\llog{D_1}$. In terms of the matrix $B$ above,
 	$\psi_1=\sum\limits_{j>1} b_{1j}\dlog(x_j)$. Note that $\psi_1$ which corresponds to the
 	Hamiltonian vector field $v_1$, is a closed form. Consequently,
$\psi_1$  defines a foliation on $X_1$. Let $Q_1^\bullet=\psi_1\Omega_{X_1}^\bullet$
	be the associated foliated De Rham complex $\psi_1\Omega_{X_1}^\bullet$:
	\[Q_1^0=\O_{X_1}\phi_1\to Q_1^1=\psi_1\Omega^1_{X_1}\simeq\Omega^1_{X_1}/\O_{X_1}\psi_1\to...
	\to Q_1^i=\wedge^iQ_1^1\to...\]
 	 endowed with the foliated differential.
 	
 	\par
 	Note that the residue exact sequence \eqref{log-res} induces the Poincar\'e residue
 	sequence
 	\[\exseq{\Omega^\bullet_{X_1}\llog{D_1}}{f_1^*\Omega^\bullet_X\llog{D}}{\Omega^\bullet_{X_1}
 		\llog{D_1}[-1]}.\]
 	Again the Poincar\'e residue of a closed form is closed.
 	Now the exact sequence
 	\[\exseq{T_X\mlog{D}}{T_X}{f_{1*}N_{f_1}}\]
 	yields
 	\eqspl{log-logplus}{\exseq{\Omega_X\llog{D}}{\Omega_X\llogp{D}}{f_{1*}N_{f_1}}.
 	} and this sequence induces the $\mathcal F_\bullet$ filtration on the
 log-plus complex $\Omega_X^\bullet\llogp{D}$.\par
\subsection{First graded piece}
 
 Now consider first the first graded 
 $\mathcal G^\bullet_1=(\mathcal F^\bullet_1/\mathcal F^\bullet_0)[1]$
 which is supported in codimension 1.
 (the shift is so that $\mathcal G^\bullet$ starts in degree 0).
 Then $\mathcal G_1^\bullet$ is a (finite) direct image of a complex of $X_1$ modules:
 \[\mathcal E_1: N_{f_1}\to N_{f_1}\otimes Q_1\to N_{f_1}\otimes Q_1^2\to ...\]
 Using Lemma \ref{normal-log-complex-lem}, we can easily show:
 \begin{prop}
 	$\mathcal E_1$ is isomorphic  to $P^\bullet_{R_1'}(N_{f_1})$, hence is 
 	null-homotopic and exact, hence $\mathcal G^\bullet_1$ is exact.
 	\end{prop}
\subsection{Second graded piece}
 	Next we study $\mathcal G_2$, which is supported on $X_2$.
We consider a connected, nonempty
 open subset $W\subset X_2$, for example an entire component, over which the
 	'normal orientation sheaf'
 	$\nu_2: X_{2,1}\to X_2$, i.e.
 	 the local $\Z_2$-system   of branches of $X_1$ along
 	$X_2$, is trivial (we can take $W=X_2$ if, e.g.  $D$ has global normal crossings).
 	Such a subset $W$ of $X_{2}$ is  said to be a \emph{normally split} subset of $X_2$
 	and a \emph{normal splitting} of $W$ is an ordering of the branches is specified.
 	Obviously $X_2$ is covered by such subsets $W$. Likewise, for a subset $Z\subset X_k$.
 	\subsubsection{Iterated residue}
 	Over a normally split subset $W$, we have a diagram
 	\eqspl{res2}{
 	\begin{matrix}
 		0\to 2\O_{W}\stackrel{\check R_2}{\to}&f_2^*T_X\mlog{D}|_W&\to T_{X_2}\mlog{D_2}|_W\to 0\\
 		&\downarrow&\\
 		0\to\Omega_{W}\llog{D_2}\to&f_2^*\Omega_X\llog{D}|_W&\stackrel{R_2}{\to} 2\O_{W}\to 0
 			\end{matrix}
 			}\
 		where $\check R_2$ is the map induced by $\check R_1$.
 		The composite map $R_2\check R_2:2\O_{W}\to 2\O_{W}$ is just the alternating form 
 		induced by $\Phi$, and has the form $c_WH_2$ where $H_2$ is the hyperbolic plane
 		$\left( \begin{matrix}0&1\\-1&0\end{matrix}\right)$.
 		In terms of a local frame for $\Omega_X\llog{D}$ containing $\dlog(x_1),
 		\dlog(x_2)$, 
 		$c_W$ is  the coefficient of $\dlog(x_1)\wedge\dlog(x_2)$ in $\Phi$. 
 		Note $c_W$ must be constant because $\Phi$ is closed.
 		In fact we have
 		\[c_W=\rmRes_1\rmRes_2(\Phi)\]
 		where $\rmRes_i$ denote the (Poincar\'e) residues along the branches of $X_1$
 		over $X_2$. Set 
 		\[\rmRes_{W}(\Phi):=c_W.\] This is essentially what is
 		called the biresidue by Matviichuk et al., see \cite{pym-loctorelli}. 
 		Thus, when $c_W\neq 0$, we have a basis for
 		the log forms 
 		\[\eta_1=\dlog(x_1),...,\eta_m=\dlog(x_m), \eta_{m+1}=dx_{m+1},...,\eta_{2}n=dx_{2n}\]
 		 		$m=$ multiplicity of $D$, $m\geq 2$, and then
 		\[\Phi=\sum b_{ij}\eta_i\wedge\eta_j\]
 		where
 		\[b_{12}=-b_{21}=c_W.\]
 		If $W$  may be not be normally orientable  (e.g. an entire component of $X_2$)
 		then $c_W$ is defined only up to sign; if $c_W=0$ 
 		we say that $W$ is non-residual, otherwise it is residual.
 		\par \subsubsection{Non-residual case}
 	Here we consider the case $c_W=0$.\par
 	Note that in that case we may express $\Phi$ along $W$ in the form
 	\[\Phi=\dlog(x_1)\gamma_3+\dlog(x_2)\gamma_4+\gamma_5\]
 	where the gammas are closed log forms in the coordinates on $W$, i.e. $x_3,...,x_{2n}$.
 	Moreover, $\gamma_3\wedge\gamma_4\neq 0$ because $\Phi^n$ is
 	divisible by $\dlog(x_1)\dlog(x_2)$.
 	Also, unless $\gamma_3, \gamma_4$ are both holomorphic (pole-free),
 	there is another component $W'$ of $X_2$ such that $c_{W'}\neq 0$
 	(in particular, $W\cap D_2\neq\emptyset$). 
 	Hence if no such $W'$ exists, 
 	we may  by suitably modifying coordinates,
 	assume locally that $\gamma_3=dx_3, \gamma_4=dx_4$.
 	A similar argument,
 	or induction, applies to $\gamma_5$. 
 	This means we are essentially in the P-normal case considered in \cite{qsymplectic}.
 	This we conclude: 
 	\begin{lem}
 		Unless $\Pi$ is P-normal, there exists a nonempty residual  open subset
 	$W$ of $X_2$.
 	\end{lem}
 		


\subsubsection{Residual case: identifying $\mathcal G_2$}
Next we analyze  a residual normally oriented open subset $W\subset X_2$.			
As above, we get a composite map of $R'_2:2\O_{W}\to f_2^*\Omega_X\llog{D}|_W$
 	, whose image we denote by $M_{2W}$. It has a local basis 
 		$(\psi_{11}=x_1\phi_1, \psi_{12}=x_2\phi_2)$
 		corresponding to the basis $(e_1, e_2)$ of $2\O_{W}$. 
 		In term of the $B$-matrix, we have
 		\[\psi_{11}=\sum b_{1j}\eta_j=-\sum b_{j1}\eta_j, 
 		\psi_{12}=-\sum b_{2j}\eta_j=\sum b_{j2}\eta_j.\]
 		As $\psi_{11}, \psi_{12}$ are closed, $M_2$ is integrable.
 	 Let $\bar\Omega$
 		denote the quotient $f_2^*\Omega_X\llog{D}|_W/M_{2W}$. Then we have an 
 		isomorphism 
 		\eqspl{omegabar}{\bar\Omega\to \Omega_{W}\llog{D_2}
 		} given explicitly
 		by\[\bar\omega\mapsto \omega- \Res_1(\omega)\psi_{12}/c_W-\Res_2(\omega)\psi_{11}/c_W\]
 		(because $\Res_2(\psi_{11})=\Res_1(\psi_{12})=c_W$, residues with respect
 		to the two branches of $D$).
 Now set $N_2=\det N_{f_2}$, an invertible sheaf on $X_2$. Then $\mathcal G^\bullet_2=
 (\mathcal F^\bullet_2/\mathcal F^\bullet _1)[2]$ is the direct image of a complex
 on $X_2$:
 \eqspl{}{
 \mathcal E^\bullet _2: N_2\to N_2\otimes\bar\Omega\to N_2\otimes\wedge^2\bar\Omega\to...	
 }
where a local generator of $N_2$ has the form $1/x_1x_2$ and the differential has the form
\[\bar\omega/x_1x_2\mapsto d\bar\omega/x_1x_1\pm (\bar\omega/x_1x_2)\dlog(x_1x_2).\]
\subsubsection{Zeroth differential}
Using the identification \eqref{omegabar}, the zeroth differential has the form
\eqspl{g}{
\tilde d(g/x_1x_2)=\frac{1}{x_1x_2}(dg+g(\dlog(x_1x_2)- (\psi_{11}+\psi_{12})/c_W)),
g\in\O_{X_2}.
}
The form $\psi_2=-\dlog(x_1x_2)+ (\psi_{11}+\psi_{12})/c_W$ has zero residues with
respect to $x_1, x_2$, hence yields a form in $\Omega_{X_2}\llog{D_2}$.
Changing the local equations $x_1, x_2$ changes $\psi$ by adding a
holomorphic (pole-free) form on $X_2$.\par
For $g$ nonzero \eqref{g} can be rewritten
\eqspl{gg}{
\tilde d(g/x_1x_2)=\frac{g}{x_2x_2}(\dlog(g)-\psi_2)	
} 
When does this operator have a nontrivial kernel? 
First, if $g$ is constant then $\psi_2=0$ on $W$ which is im[possible
if $W$ meets $D_2$.
Next, locally at a point $x\in W\setminus D_2\cap W$,  clearly $g/x_1x_2$
holomorphic and nonzero in the kernel exists locally since $\psi_2$
is closed and holomorphic so $\psi_2=dh$ for a holomorphic function $h$
 and we can take $g=e^h$.
Moreover nonzero solutions to $d(g/x_1x_2)=0$  differ by a multiplicative constant.
The condition that the local solutions patch is clearly that
$\frac{1}{2\pi i}\int\limits_\gamma \psi_2$ be an integer for any
loop $\gamma$ in $W\setminus D_2\cap W$.
Now $\psi_2$ is defined only modulo a holomorphic form on $X_2$
while $H_1(W\setminus D_2\cap W)$ is generated modulo $H_1(W)$
 by small loops normal to components
of $D_2$, So the relevant condition is just integrality over such loops $\gamma$.\par 
At a simple point of $D_2\cap W$, the condition that $g$ exist locally as a holomorphic
function with no pole on $D_2$ is clearly that for $\gamma$ as above,
oriented positively, 
 the integer $\frac{1}{2\pi i}\int\limits_\gamma \psi_2$ is nonnegative, so that $g$ has
 no pole on $D_2$.
In other words, that the sum of the first 2 columns of the $B$ matrix,
normalized so that $b_{12}=-b_{21}=1$, should be a nonnegative integer vector.
Finally by Hartogs, if $g$ is holomorphic off the singular locus of $D_2\cap W$, it extends 
holomorphically to $W$.\par
\subsubsection{Special components}
Now let $Z$ be a component of $D_2\cap W$ and assume $W$ and $Z$ are both normally split
so that the branches of $D$ along $W$ may be labelled 12 while those along $Z$ may be labelled
123. Thus branches of $X_2$ over $Z$ are labelled 12, 23, 31 and the preceding discussion
shows that the zeroth differential has nontrivial kernel along $Z$ only if
the iterated residues of $\Phi$ along these branches, denoted $c_{21}, c_{23}, c_{31} $,
assuming $c_{12}\neq 0$, satisfy
\eqspl{special}{
	c_{23}+c_{31}=kc_{21}, k\in\N.
}
We call such a component $Z$ \emph{special}; then $W$ is said to be special if every (normally split)
component of $D_2\cap W$ is special.\par
What about the normally split hypothesis? Suppose first $W$ is contained
in a connected open set $W'$ which is not normally split. 
Then as $c_{12}$ is locally constant in $W'$ it follows that $c_{12}=0$, i.e. $W$ is not residual.
Now suppose $Z$ is contained in $Z'$ open connected and not normally split. 
Then monodromy acts on the branches of $X_2$ along $Z'$ cyclically
and consequently the $c_{ij}$ above are all equal. Then \eqref{special} holds automatically with $k=2$
so $Z$ is special.\par

\subsubsection{Conclusion}
What we have so far proven is the following: if
$W$ is a normally oriented residual  open subset of of $X_2$ 
 then the stalk of the zeroth cohomology $\H^0(\mathcal G_2^\bullet)$
vanishes somewhere on $W$ \emph{unless} either\par
(i) $W\cap D_2=\emptyset$, or\par
(ii)  $W$ is special.\par
	

 Note that if the stalk of  $\H^0(\mathcal G_2^\bullet)$
 vanishes somewhere in $W$, then 
 because $\mathcal G^0_2$ is coherent and torsion-free, 
 it follows that $H^0(\mathcal G^\bullet_2)|_W=0$,
 hence a similar vanishing holds for the entire component of $X_2$
 containing $W$. Now recall that, minding the index shift, if
 $H^0(\mathcal G^\bullet_2)=0$ then the cokernel of the inclusion
 $\Omega^{+\bullet}_X\llog{D}\to \Omega^{+\bullet}_X\llogp{D}$
 has vanishing $\hh^1$ (and $\hh^0$). On the other hand, it is well known
 (see e.g. \cite{logplus}) that $\Omega_X^{+\bullet}\llog{D}\simeq T_X^\bullet\mlog{D}$
 controls deformations of $(X, \Phi)$ or $(X, \Pi)$ where $D$ deforms locally trivially,
 and those deformations are unobstructed thanks to Hodge theory.

Summarizing this discussion, we conclude:
\begin{thm}\label{iso}
	Let $(X, \Phi)$ be a log-symplectic manifold with polar divisor $D$.
With notations as above, let 
\[\Omega^{+\bullet}_X\llog{D}=\bigoplus\limits_{i>0}\Omega^i_X\llog{D}, 
\Omega^{+\bullet}_X\llogp{D}=\bigoplus\limits_{i>0}\Omega^i_X\llogp{D}.\]

Then
 the inclusions	
\[\Omega_X^{+\bullet}\llog{D}\to\Omega_X^{+\bullet}\llogp{D},\]
\[ T_X^\bullet\mlog{D}\to T^\bullet_X\] induce
isomorphisms on $\H^2$ and  injections on $\H^3$,  hence 
isomorphisms 
on $\hh^1$ and injections on $\hh^2$, \ul{unless}
either\par (i) $X_2$ has a non-residual component; or\par
(ii) $X_2$ has a special component.
\end{thm}
As noted above, any component of $X_2$ that is disjoint from $D_2$, i.e. contains no
triple points of $D$, is automatically non-residual.
\begin{cor}
	Notations as above, if $X$ is compact and K\"ahlerian and
	 conditions (i), (ii) both fail, then the pair $(X, \Phi)$
	has unobstructed deformations and the polar divisor of $\Phi$ deforms locally trivially.
	\end{cor}
In the case where $D$ has global normal crossings, i.e. is a union of smooth divisors, 
this result also follows from results in \cite{pym-loctorelli}, which also states a partial converse: when 
$T_X^\bullet\mlog{D}\to T^\bullet_X$ is not a quasi-isomorphism, $(X, \Phi)$ has obstructed deformations
and admits deformations  where $D$ either smooths or deforms locally trivially.
\begin{example}(Due to M. Matviichuk, B. Pym, T. Schedler, see \cite{pym-loctorelli}, communicated by \
	B. Pym)	\label{pym-example}
	Consider the matrix
	\eqspl{}{
		B=(b_{ij})=\left (
		\begin{matrix}
			0&1&2&4\\
			-1&0&3&5\\
			-2&-3&0&6\\
			-4&-5&1&0
		\end{matrix}	\right)	
	}
	and the corresponding log-symplectic form on $\C^4$, $\Phi=\sum \limits_{i<j}
	b_{ij}\frac{dz_i}{z_i}\wedge \frac{dz_j}{z_j}$ and corresponding Poisson
	structure $\Pi=\Phi\inv$, both of which extend to $\P^4$ with Pfaffian divisor
	$D=(z_0z_1z_2z_3z_4)$, $z_0=$ hyperplane at infinity. 
	Then $\Pi$ admits the 1st order Poisson deformation
	with bivector $z_3z_4\del_{z_1}\del_{z_2}$, which in fact extends to
	a Poisson deformation of $(\P^4, \Pi)$
	over the affine line $\C$, and the Pfaffian divisor deforms 
	as $(z_3z_4z_0(z_1z_2-tz_3z_4))$, hence non locally-trivially.
	Correspondingly, the log-plus form $z_3z_4\phi_1\phi_2$ is closed (
	and not exact). That  $d(z_3z_4\phi_1\phi_2)=0$ corresponds to the 
	integral column relation
	\[k_1-k_2+(e_1+e_2)-(e_3+e_4)=0\]	
	where the $k_i$ and $e_j$ are the columns of 
	the $B$ matrix and the identity, respectively, showing that $(z_1z_2z_3)$ and
	$(z_1z_2z_4)$ are residual triples of type II and (12), i.e. $(x_1)\cap (x_2)$ 
	is a special
	component of $X_2$.
\end{example}
\begin{rem}
	As we saw above, the presence of monodromy on the branches of $D$ is related to non-residual or
	special components. This suggests that log-symplectic manifolds with irreducible polar
	divisor may often be obstructed. However we don't have specific examples.
	\end{rem}

\vfill\eject
\bibliographystyle{amsplain}
\bibliography{../../../mybib}
\end{document}

%% file: myheader3.tex
 \DeclareMathOperator{\Res}{Res}

\def\inv{^{-1}}

\DeclareMathOperator{\del}{\partial}
\DeclareMathOperator{\Ext}{Ext}

\DeclareMathOperator{\dlog}{dlog}

\def\refp #1.{(\ref{#1})}
\newcommand\carets [1]{\langle #1 \rangle}

\newcommand{\ul}[1]{\underline {#1}}

\def\sbr #1.{^{[#1]}}
\def\sfl #1.{^{\lfloor #1\rfloor}}

\def\hh{\mathbb H}

\def\inv{^{-1}}
\def\?{{\bf{??}}}

\def\H{\mathcal H}

\def\C{\mathbb C}
\def\P{\mathbb P}
\def\N{\mathbb N}

\def\Z{\mathbb Z}

\def\O{\mathcal O}

\def\id{\text{id}}

\def\g{\mathfrak g}

\def\1/2{\frac{1}{2}}

\def\I{\mathcal{ I}}

\def\2{{[2]}}
\def\l{\ell}
\def\nl{\newline}

\def\<{\langle}
\def\>{\rangle}

\def\2{{[2]}}
\def\l{\ell}

\def\scl #1.{^{\lceil#1\rceil}}
\def\spr #1.{^{(#1)}}
\def\sbc #1.{^{\{#1\}}}

\def\subpr#1.{_{(#1)}}

\def\beq{\begin{equation*}}
\def\eeq{\end{equation*}}

\newcommand{\llog}[1]{\langle\log {#1} \rangle}
\newcommand{\llogp}[1]{\langle\log  {^+#1} \rangle}

\newcommand{\mlog}[1]{\langle-\log {#1} \rangle}
\def\g3{{\Gamma\spr 3.}}

\newcommand{\eqspl}[2]{
\begin{equation}\label{#1}
\begin{split}
#2\end{split}\end{equation}}

\newcommand{\exseq}[3]{
0\to #1\to #2\to #3\to 0
}

\newcommand{\beginalphaenum}{
\begin{enumerate}\renewcommand{\labelenumi}{ }
\item \begin{enumerate}
}

\def\eex{\end{rm}\end{example}}